 \documentclass[final,3p,times]{elsarticle}

\usepackage{amssymb}
\usepackage{amsthm}
\usepackage[english]{babel} \usepackage{latexsym,amsfonts,amsmath,mathrsfs}
\newtheorem {theorem}{Theorem}[section]         \newtheorem {lemma}[theorem]{Lemma}     \newtheorem {definition}[theorem]{Definition}
     \newtheorem {remark}[theorem]{Remark}   
       \newtheorem {proposition}[theorem]{Proposition} 

\newcommand{\C}{\mathbb C}
\newcommand{\R}{\mathbb R}

\newcommand{\Di}{\mathbb{D}}
\newcommand{\Dr}{\mathbb{D}_R}

\newcommand{\gm}{\gamma}
\newcommand{\norm}[1]{\left\Vert#1\right\Vert}
\newcommand{\scal}[1]{\left<#1\right>}
\newcommand{\mes}{d\lambda}
\newcommand{\muar}{d\mu_{\alpha,R}}
\newcommand{\nuar}{d\mu_{\nu}}

\newcommand{\nuarexp}{e^{-\nu \vert z\vert^{2}}  \mes}
\newcommand{\LDo}{L^{2}(\Dr;\muar)}
\newcommand{\LC}{L^{2}(\C;\nuarexp)}

\newcommand{\HolD}{\mathcal{H}ol(\Di)}
\newcommand{\HolDr}{\mathcal{H}ol(\Dr)}
\newcommand{\HolC}{\mathcal{H}ol(\C)}
\newcommand{\ADm}{\mathcal{A}^{2,\alpha}_{R,m}(\Dr)}
\newcommand{\ADmm}{\mathcal{A}^{2,\alpha}_{R,m+1}(\Dr)}

\newcommand{\ADmo}{\mathcal{A}^{2,\alpha}_{R}(\Dr)}
\newcommand{\BDm}{\mathcal{B}^{2,\nu}_{m}(\C)}
\newcommand{\BDmm}{\mathcal{B}^{2,\nu}_{m+1}(\C)}
\newcommand{\BDmmo}{\mathcal{B}^{2,\nu}_{1}(\C)}
\newcommand{\BDmo}{\mathcal{B}^{2,\nu}(\C)}

\begin{document}

\begin{frontmatter}



\title{Weighted Bergman-Dirichlet and Bargmann-Dirichlet spaces of order $m$: Explicit formulae for reproducing kernels and asymptotic}
\author{A. El Hamyani}  \ead{amalelhamyani@gmail.com}
\author{A. Ghanmi}      \ead{ag@fsr.ac.ma}
\author{A. Intissar}    \ead{intissar@fsr.ac.ma}
\author{Z. Mouhcine}    \ead{zakariyaemouhcine@gmail.com}
\author{M. Souid El Ainin}   \ead{msouidelainin@yahoo.fr}
 \address{E.D.P. and Spectral Geometry,
          Laboratory of Analysis and Applications-URAC/03,
          Department of Mathematics, P.O. Box 1014,  Faculty of Sciences,
          Mohammed V University, Rabat, Morocco}

\begin{abstract}
We introduce new functional spaces that generalize the weighted Bergman and Dirichlet spaces on the disk $D(0,R)$ in the complex plane
 and the Bargmann-Fock spaces on the whole complex plane. We give a complete description of the considered spaces. Mainly, we are interested in giving
explicit formulas for their reproducing kernel functions and their asymptotic  behavior as $R$ goes to infinity.
\end{abstract}


\end{frontmatter}


\section{Introduction and  main results}
Let $\Di$ be the unit disk in the complex plane $\C$ and denote by $\mathcal{D}_\gm$; $\gm \in \R$, the functional space of all analytic functions $f(z) = \sum\limits_{n=0}^{+\infty} a_nz^n$
on $\Di$ such that its norm $\norm{f}_\gm$ is finite, where
$$
\norm{f}_\gm^{2} := \sum\limits_{n=0}^{+\infty} (n + 1)^\gm |a_n|^2 .
$$
 Thus for special values $\gm=-1$, $\gm=0$ and $\gm=1$  we have the Bergman, Hardy and Dirichlet space, respectively.
 More general, $\mathcal{D}_\gm$ is a weighted Bergman space when $\gm<0$ and a weighted Dirichlet space when $\gm>0$.
Such spaces play important roles in function theory and operator theory, as well as in modern analysis, probability and statistical analysis.
For a nice introduction and surveys of these spaces in the context function and operator theories, see \cite{RichterShields1988,RochbergWu1993,Wu1998,
ArcozziRochbergSawyerWick2011} and the references therein.

Added to the sequential characterization, the weighted Bergman space can be described differently. It can be realized as the
$(1-|z|^2)^\alpha\mes$ square integrable functions on $\Di$ that are holomorphic on $\Di$
\begin{align}\label{wBs}
\mathcal{A}^{2,\alpha}(\Di) := L^{2}\left(\Di; \left(1-\left|z\right|^{2}\right)^{\alpha}d\lambda\right) \cap \HolD,
\end{align}
where $\mes(z) = dx dy=\frac i2 dz\wedge d\bar z$ with $z=x+iy$; $x,y\in\R$,
 is the two dimensional Lebesgue area measure.
The corresponding reproducing kernel is known to be given through 
\begin{align} \label{RK-wBs}
K^{\alpha}(z,w)= \frac{(\alpha+1)}{\pi} \left( \frac{1}{1 -  z \overline{w}} \right)^{\alpha}.
\end{align}
While the classical Dirichlet space can be defined as the class of analytic functions $f(z) = \sum\limits_{n=0}^{+\infty} a_nz^n$ on $\Di$ for which
the semi-norm defined by the Dirichlet integral
$$ D(f) := \int_{\Di} |f'(z)|^2 \mes(z) = \sum\limits_{n=0}^{+\infty} n|a_n|^2 $$
is finite. A More convenient norm to use on the classical Dirichlet space  is the following
$$ \norm{f}_{1,0}^2 : =  |f(0)|^2 + \frac 1{\pi} \int_D |f'(z)|^2 \mes(z).$$
The reproducing kernel of the classical Dirichlet space with respect to this norm is known to be given by \cite{ArcozziRochbergSawyerWick2011}
\begin{align}\label{RKcD}
K(z;w) = \frac 1\pi \left(1 + \log\left(\frac{1}{1 - z\overline{w}}\right)\right); \quad z,w \in \Di.
\end{align}

In this paper, we intend to introduce and study two new classes of functional spaces on the disc $\Dr$ of radius $R$ as well as on $\C$, labeled by some fixed nonnegative integer $m$.
The first one on  $\Dr$ is denoted $\ADm$, $\alpha > -1$, and called weighted Bergman-Dirichlet space of order $m$. It generalizes the classical Bergman and Dirichlet spaces,
and consists of all convergent entire series on $\Dr$ whose the norm $\norm{\cdot}_{\alpha,m,R}$ defined through \eqref{Norm-m} below is finite.
Our main results concerning $\ADm$ are summarized in the following

\begin{theorem}\label{Mthm1}
The space $\ADm$ is non trivial if and only if $\alpha > -1$. In this case $\ADm$ is a reproducing kernel Hilbert space. Its reproducing kernel
is given explicitly in terms of the ${_3F_2}$-hypergeometric function as
\begin{align} \label{RK-m}
K^{\alpha}_{R,m}(z,w) = \frac{(\alpha+1)}{\pi R^{2}}  \left\{\sum\limits_{n=0}^{m-1}\frac{\Gamma(\alpha+2+n)}{\Gamma(n+1)\Gamma(\alpha+2)}\left(\frac{z\overline{w}}{R^{2}}\right)^{n}
+
 \frac{ (z\overline{w})^{m}}{(m!)^2} {_3F_2}\left( \begin{array}{c} 1,1 ,\alpha+2\\m+1,m+1\end{array}\bigg| \frac{z\overline{w}}{R^{2}} \right) \right\}.
\end{align}
Moreover, a function $f(z)=\sum\limits_{n=0}^{+\infty}a_{n}z^{n} $ on $\Dr$ belongs to $\ADm$ if and only if
   \begin{align}
\norm{f}_{\alpha,m,R}^{2} = \pi\left\{ \sum\limits_{n=0}^{m-1} \left(         \frac{n!\Gamma(\alpha+1)R^{2n+2}}{\Gamma(n+\alpha+2)}\right) |a_{n}|^{2}
+ \sum\limits_{n=m}^{\infty} \left(    \frac{(n!)^2\Gamma(\alpha+1) R^{2(n-m)+2}}{(n-m)!\Gamma(n-m+\alpha+2)}\right) |a_{n}|^{2} \right\}  < +\infty. \nonumber
  \end{align}
\end{theorem}

\begin{remark}\label{Rem-1}
\begin{enumerate}
\item The special case of $R=1$ and $m=0$ leads to the weighted Bergman space \eqref{wBs}. In this case, the expression  of the reproducing kernel in \eqref{RK-m} reduces further to the Bergman reproducing kernel \eqref{RK-wBs}.

\item For $R=1$, $\alpha=0$ and $m=1$, the corresponding space is the classical Dirichlet space. In this case
the expression  of the reproducing kernel \eqref{RK-m} reduces further to the reproducing kernel given through \eqref{RKcD} of the classical Dirichlet space.
\end{enumerate}
\end{remark}

What we do in the construction of $\ADm$ works mutatis mutandis to introduce and study their analogues on the whole complex plane $\C$, the Bargmann-Dirichlet spaces $\BDm$ of order $m$  (see Section 3). The following is the analogue of Theorem \ref{Mthm1} for these spaces.

\begin{theorem}\label{Mthm2}
The space $\BDm $ is a reproducing kernel Hilbert space.
   Its reproducing kernel function is given in terms of the ${_2F_2}$-hypergeometric function by
\begin{align}\label{RK-BDC} K_{m}^{\nu}(z,w)=\frac{\nu}{\pi}\left(
  \sum\limits_{k=0}^{m-1}\frac{(\nu z \overline{w})^{k}}{k!}+\frac{(z \overline{w})^{m}}{\Gamma^{2}(m+1)} \,
   {_2F_2}\left(\begin{array}{c} 1, 1\\ m+1,m+1 \end{array}\bigg |  \nu z \overline{w}\right)\right) .
   \end{align}
   Moreover, a function $f(z)=\sum\limits_{n=0}^{+\infty}a_{n}z^{n} $ belongs to $\BDm$ if and only if
   \begin{align} \label{normefc}
     \norm{f}_{\nu,m}^{2}=\pi \left\{ \sum\limits_{n=0}^{m-1}
     \left(\frac{n!}{\nu^{n+1}}\right)\mid a_{n}\mid ^{2} +\sum\limits_{n=m}^{+\infty}\left( \frac{(n!)^{2}}{\nu^{n-m+1}(n-m)!}\right)\mid a_{n}\mid ^{2} \right\}<+\infty .
   \end{align}
\end{theorem}

\begin{remark}\label{Rem-2}
For $m=0$, the space $\BDm $ is the Bargmann-Fock space,  consisting of holomorphic functions on $\C$ that are $e^{\nu|z|^2}\mes$-square integrable.
Its reproducing kernel function is known to be given by 
  $$ K^{\nu}(z,w)=\left(\frac{\nu}{\pi}\right) e^{\nu z \overline{w}}.$$
\end{remark}

Motivated by the fact that the flat hermitain geometry on $\C$
can be approximated by the complex hyperbolic geometry of the disks $\Dr$ of radius $R>0$
associated to an appropriate scaled Bergman K\"ahler metric \cite{GhIn2005JMP}  (see Section 4),
we show that the spaces $\BDm$, with $\alpha = \nu R^{2}$, can be seen as the limit of the spaces $\ADm$ as $R$ goes to infinity, in the sense that we have

\begin{theorem}\label{Mthm3}
For every fixed nonegative integer $m$, the reproducing kernel
 $K^{\nu R^2}_{R,m}(z,w)$ of the weighted Bergman-Dirichlet space $\ADm$ converges
 pointwisely and uniformly on compact sets of $\C\times \C$ to the reproducing kernel function of
weighted Bargmann-Dirichlet space $\BDm$.
\end{theorem}

The paper is organized as follows. In the succeeding sections (Sections 2 and 3), we discuss the proofs of our main results, Theorems \ref{Mthm1} and  \ref{Mthm2}, stated in this introductory section. Moreover, we give a complete description of the considered Hilbert spaces $\ADm$ and $\BDm$, including the explicit formulae for their reproducing kernel functions.
In Section 4, we show that the $L^2$-eigenprojector kernel of $\ADm$ on $\Dr$ gives rise to its analogue of $\BDm$ on $\C$ by letting $R$ tends to infinity.

\section{Weighted Bergman-Dirichlet spaces of order $m$ on the disk $\Dr$}

Denote by $\Dr=\{z\in \C; \, |z|<R\}$ the disk of radius $R>0$ in the complex plane $\C$.
For given $\alpha\in\R$, let  $\LDo:=L^{2}(\Dr; d\mu_{\alpha,R})$ be the space of complex valued functions on
$\Dr$ that are square-integrable with respect to the density measure
\begin{align}\label{measure}
\muar (z)=\left(1-\left|\frac{z}{R}\right|^{2}\right)^{\alpha}\mes(z), 
\end{align}
$\mes$ being the two dimensional Lebesgue area measure on $\C$.
 The space $\LDo$ is a Hilbert space in the norm
 \begin{align}\label{FreeNorm}
\|f \|^{2}_{\alpha,R}:= \int_{\Dr}|f(z)|^{2} \, \muar(z)
\end{align}
corresponding to the hermitian scaler product
\begin{align}\label{FreeSP}
\scal{ f , g}_{\alpha,R}:= \int_{\Dr} f(z) \overline{g(z)}  \muar(z).
\end{align}
By $\HolDr$, we denote the vector space of all convergent entire series $f(z)=\sum\limits_{n=0}^{+\infty} a_n z^n$ on $\Dr$.
Note that, for a given arbitrary nonnegative integer $m=0,1,2, \cdots$, we can split any $f \in \HolDr$ as
\begin{align}\label{Split-m}
 f(z) = f_{1,m}(z) + f_{2,m}(z) ,
 \end{align}
 where
$$ f_{1,m}(z) := \sum\limits_{n=0}^{m-1} a_n z^n  \quad \mbox{and} \quad f_{2,m}(z) := \sum\limits_{n=m}^{+\infty} a_n z^n = f(z) - f_{1,m}(z)$$
so that $f^{(m)}= f_{2,m}^{(m)}$, with the convention that $ f_{1,0}(z) =0$ when $m=0$.
Thus for any fixed nonnegative integer $m$, we consider the functional space $\ADm$ of all $f \in \HolDr$ such that
\begin{align} \label{Norm-m}
\norm{f}_{\alpha,m,R}^2 : = \norm{ f_{1,m} }_{\alpha,R}^2 + \norm{ f_{2,m}^{(m)}}_{\alpha,R}^2 < +\infty.
\end{align}
We denote by $\scal{ , }_{\alpha,m,R}$ the associated hermitian scalar product defined by
\begin{align} \label{SP-m}
\scal{ f , g}_{\alpha,m,R} : = \scal{ f_{1,m}  ,  g_{1,m} }_{\alpha,R} + \scal{ f_{2,m}^{(m)} , g_{2,m}^{(m)}}_{\alpha,R}
\end{align}
for given $f,g\in \ADm$.

The aim of this section is to give a concrete description of $\ADm$ and prove Theorem \ref{Mthm1}. We begin with the following
\begin{lemma}\label{lem-orth-norm}
Keep notations as above.
\begin{enumerate}
  \item[(i)] The monomials $e_{n}(z)=z^n$ are pairwise orthogonal with respect to
the hermitian scalar product $\scal{ , }_{\alpha,m,R}$ in \eqref{SP-m}.
  \item[(ii)] The monomials $e_{n}(z)=z^n$ belong to $\ADm$ if and only if $\alpha>-1$.
  \item[(iii)] For $\alpha>-1$, we have
   \begin{align} \label{nk-sp}
\norm{ e_{n}}_{\alpha,m,R}^2
= \pi \left\{ \begin{array}{ll} R^{2n+2}        \frac{n!\Gamma(\alpha+1)}{\Gamma(n+\alpha+2)}         & \quad\mbox{for } \, n < m \\
     \\R^{2(n-m)+2}    \frac{(n!)^2\Gamma(\alpha+1)}{(n-m)!\Gamma(n-m+\alpha+2)} & \quad \mbox{for } \, n\geq m
\end{array} \right. .
 \end{align}
\end{enumerate}
\end{lemma}

\begin{proof} For (i), we distinguish three cases. Indeed, we have
\begin{align}
\scal{ e_{n} , e_{k} }_{\alpha,m,R}
&= \left\{ \begin{array}{ll}
\scal{ e_{n} , e_{k} }_{\alpha,R}          & \quad \mbox{for } \, n,k < m \\
\\
0 & \quad \mbox{for } \, n<m, \, k \geq m\\
\\
\frac{n!k!}{(n-m)!(k-m)!} \scal{ e_{n-m} , e_{k-m} }_{\alpha,R} & \quad \mbox{for } \,n,k \geq m
\end{array} \right.  . \label{int-sp1}
 \end{align}
This reduces further to the computation of $\scal{ e_{n} , e_{k}}_{\alpha,R}$, which can be handled using polar coordinates $z=r Re^{i\theta}$ with
$r\in [0,1[$ and $\theta\in [0,2\pi[$. Thus, we have
\begin{align*}
\scal{ e_{n} , e_{k}}_{\alpha,R}
& = \int_{\Dr} z^n \overline{z}^k  \left(1-\left|\frac{z}{R}\right|^{2}\right)^{\alpha}d\lambda(z)\\
& = \int_{[0,1[\times[0,2\pi[} \left(rR\right)^{n+k} e^{i(n-k)\theta}   (1-r^{2})^{\alpha}rR^2drd\theta.
\end{align*}
By means of Fubini-Tonelli theorem, we get
 \begin{align}
\scal{ e_{n} , e_{k}}_{\alpha,R}
& =  R^{n+k+2} \int_0^1  r^{n+k+1} (1-r^{2})^{\alpha}\left(\int_0^{2\pi} e^{i(n-k)\theta} d\theta \right) dr \nonumber\\
& =  2\pi R^{n+k+2} \left(\int_0^1  r^{n+k+1} (1-r^{2})^{\alpha} dr \right) \delta_{n,k} . \label{toOrth}
 \end{align}
 Whence in view of \eqref{int-sp1}, we conclude that $\scal{ e_{n} , e_{k}}_{\alpha,m,R}=0$  for $n\ne k$.

 The proof of (ii) follows by taking $k=n$ in \eqref{toOrth} and next making use of the change $t=r^2$. Indeed, we obtain
 \begin{align*}
\norm{ e_{n} }_{\alpha,R}^2  = \pi R^{2n+2}  \int_0^1  t^{n} (1-t)^{\alpha}dt  .
 \end{align*}
The involved integral is then a special case of the well known Euler Beta function \cite[p. 18]{Rainville71}
$$B(x,y):=\int_{0}^{1}t^{x-1}\left(1-t\right)^{y-1} dt = \frac{\Gamma(x)\Gamma(y)}{\Gamma(x+y)} $$
 provided that $\Re(x)>0$ and $\Re(y)>0$.
 Therefore, the norm $\norm{ e_{n} }_{\alpha,m,R}$ is finite if and only if $\alpha > -1$.
 In this case, we have
 \begin{align*}
\norm{ e_{n}}_{\alpha,R}^2 = \pi R^{2n+2}  \frac{\Gamma(n+1)\Gamma(\alpha+1)}{\Gamma(n+\alpha+2)}.
\end{align*}
By substituting this in \eqref{int-sp1}, it follows
\begin{align*}
\norm{ e_{n}}_{\alpha,m,R}^2
= \left\{ \begin{array}{ll}
\pi R^{2n+2}        \frac{n!\Gamma(\alpha+1)}{\Gamma(n+\alpha+2)}         & \quad n < m \\
\\
\pi R^{2(n-m)+2}    \frac{(n!)^2\Gamma(\alpha+1)}{(n-m)!\Gamma(n-m+\alpha+2)} & \quad n\geq m
\end{array} \right. .
 \end{align*}
 Thus the proof is completed.
\end{proof}

The first main result of this section is the following
\begin{lemma}\label{lem-growth}
The space $\ADm$ is nontrivial if and only if $\alpha > -1 $. In this case, a function
$f(z)= \sum\limits_{n=0}^{\infty}a_{n}z^{n}$ belongs to $\ADm$ if and only if  $(a_n)_n$ satisfies the growth condition
$\sum\limits_{n=0}^{\infty}  |a_{n}|^{2} \norm{e_{n}}_{\alpha,m,R}^{2} < +\infty ,$
which reads explicitly as ,
 \begin{align}
\norm{f}_{\alpha,m,R}^{2} = \pi\left\{ \sum\limits_{n=0}^{m-1} \left(         \frac{n!\Gamma(\alpha+1)R^{2n+2}}{\Gamma(n+\alpha+2)}\right) |a_{n}|^{2}
+
 \sum\limits_{n=m}^{\infty} \left(    \frac{(n!)^2\Gamma(\alpha+1) R^{2(n-m)+2}}{(n-m)!\Gamma(n-m+\alpha+2)}\right) |a_{n}|^{2} \right\}
  < +\infty. 
 \end{align}
\end{lemma}

\begin{proof}
Lemma \ref{lem-orth-norm} shows that the monomials $ z^n $ belong to $\ADm$ under the assumption $\alpha>-1$.
For the converse, assume that $\ADm$ is nontrivial and pick a nonzero function $f(z)= \sum\limits_{n=0}^{\infty}a_{n}z^n$ such that
$\norm{f}_{\alpha,m,R}^{2} <+\infty$. Therefore, according to $(i)$ of Lemma \ref{lem-orth-norm}, we get
$$\norm{f}_{\alpha,m,R}^{2} 
   =  \sum\limits_{n=0}^{\infty} |a_{n}|^{2} \norm{e_{n}}_{\alpha,m,R}^{2} <+\infty.$$
This implies that $|a_{n}|^{2} \norm{e_{n}}_{\alpha,m,R}^{2}  <+\infty $ for every $n$ and in particular for certain $n_0$ for which $a_{n_0}\ne 0$. Thus from (ii) of Lemma \ref{lem-orth-norm}, we deduce that $\alpha > -1$.
\end{proof}

\begin{remark} For $R=1$ and $\alpha=0$, the spaces $\ADm$ corresponding to $m=0$ and $m=1$ can be identified respectively to
$$ \left\{(a_n)_{n\geq 0} \subset \C ; \, \sum\limits_{n=1}^{+\infty} \frac{|a_n|^2}{n+1} <+\infty \right\}
\quad \mbox{ and } \quad
   \left\{(a_n)_{n\geq 1} \subset \C ; \, \sum\limits_{n=1}^{+\infty} n |a_n|^2 <+\infty \right\}$$
which are respectively the sequential characterization of the classical Bergman and Dirichlet spaces.
\end{remark}

\begin{definition}\label{Def-1}
We will call $\ADm$, when $\alpha > -1$, the weighted Bergman-Dirichlet space of order $m$.
\end{definition}

From now on we assume that $\alpha>-1$.

\begin{lemma}\label{lem-embedding}
For every fixed $m$, the space $\ADmm$; $\alpha > -1$, is continuously embedded in $\ADm$, in the sense that
\begin{align}\label{Inegality2}
\norm{f}_{\alpha,m,R}^2 \leq   \frac 1 {C_{\alpha,R,m}}  \norm{f}_{\alpha,m+1,R}^2 .
\end{align}
for every $f\in \ADm$, where
$$ C_{\alpha,R,m} = \min\left(1, \frac{R^{2m}\Gamma(\alpha+2)}{m!\Gamma(m+\alpha+2)}, \frac{\alpha}{R^2}\right) .$$
In particular, $\ADm$ is continuously embedded in the weighted Bergman space $\ADmo={\mathcal{A}^{2,\alpha}_{R,0}(\Dr)}$.
\end{lemma}

\begin{proof} According to Lemma \ref{lem-growth}, any $f(z)= \sum\limits_{n=0}^{\infty}a_{n}z^n$ in $\ADmm$
 satisfies
$$ \norm{f}_{\alpha,m+1,R}^2= \sum\limits_{n=0}^{\infty} |a_{n}|^{2}\norm{e_{n}}_{\alpha,m+1,R}^2<+\infty; \, e_n(z)=z^n.$$
Now by means of \eqref{nk-sp}, we get
 \begin{align*}
\norm{ e_{n}}_{\alpha,m+1,R}^2
= \left\{ \begin{array}{lll}
\norm{ e_{n}}_{\alpha,m,R}^2        & \quad \mbox{for } \, n < m < m+1 \\
\\
\frac{R^{2m}\Gamma(\alpha+2)}{m!\Gamma(m+\alpha+2)} \norm{ e_{m}}_{\alpha,m,R}^2         & \quad \mbox{for } \, n =  m  < m+1\\
\\
\frac{(n-m)(n-m+\alpha+1)}{R^2}  \norm{ e_{n}}_{\alpha,m,R}^2
& \quad \mbox{for } \,  n\geq m + 1 \geq m
\end{array} \right.
 \end{align*}
 and therefore
  \begin{align*}
\norm{ e_{n}}_{\alpha,m+1,R}^2
&\geq \min\left(1, \frac{R^{2m}\Gamma(\alpha+2)}{m!\Gamma(m+\alpha+2)}, \frac{\alpha}{R^2}\right)  \norm{ e_{n}}_{\alpha,m,R}^2\\
&\geq   C_{\alpha,R,m} \norm{ e_{n}}_{\alpha,m,R}^2,
 \end{align*}
 where the constant $ C_{\alpha,R,m}$ is independent of $n$. It is given by
  $$ C_{\alpha,R,m} = \min\left(1, \frac{R^{2m}\Gamma(\alpha+2)}{m!\Gamma(m+\alpha+2)}, \frac{\alpha}{R^2}\right) .$$
 Thus, it follows
  \begin{align*}
   \norm{f}_{\alpha,m+1,R}^2 \geq  C_{\alpha,R,m}   \sum\limits_{n=0}^{\infty} |a_{n}|^{2}\norm{e_{n}}_{\alpha,m,R}^2
    \geq  C_{\alpha,R,m}  \norm{f}_{\alpha,m,R}^2.
  \end{align*}
This implies in particular that $\norm{f}_{\alpha,m,R}^2$ is finite, so that $f \in \ADm$ and the embedding mapping is continuous from $\ADmm$ into $\ADm$. This completes the proof.
\end{proof}

An other basic property for the spaces $\ADm$ is the following

\begin{proposition}\label{Hilbert}
The space $\ADm$ is a Hilbert space and the monomials $e_{n}(z)=z^{n}$; $n\geq 0$, constitute an orthogonal basis of it.
\end{proposition}

\begin{proof}
Since $\ADm$ is continuously embedded in the weighted Bergman space $\ADmo$, it is not difficult to see that $\ADm$ is a Hilbert space.
What is needed, to show that $\{e_{n}\}_n$ is a basis of $\ADm$, is completeness.
Indeed, let $f(z)= \sum\limits_{n=0}^{\infty}a_{n}z^n $ be in the orthogonal of the linear span of $(e_{n})_n$ in $(\ADm,\scal{ \cdot , \cdot}_{\alpha,m,R})$,
$$f(z)= \sum\limits_{n=0}^{\infty}a_{n}z^n  \in \Big(Span\{e_{n}; \, n\geq 0\}\Big)^{\perp_{\scal{ \cdot , \cdot}_{\alpha,m,R}}}.$$
 Thus, we have
$\scal{ f , e_{n}}_{\alpha,m,R} = 0$ for every $n$. Now, since
$$ \scal{ f, e_{n}}_{\alpha,m,R} =
a_n \left\{ \begin{array}{lll}
\norm{ e_{n}}_{\alpha,R}^2        & \quad \mbox{for } \, n < m  \\  \\
(m!)^2 \norm{ e_{n-m}}_{\alpha,R}^2       & \quad \mbox{for } \, n \geq  m
\end{array} \right. ,
$$
it follows that $a_{n}=0$ for all $n\geq 0$.
 This proves $$\{0\}= (Span\{e_{m+n}; \, n\geq 0\})^{\perp_{\scal{ \cdot , \cdot}_{\alpha,m,R}}} = (\overline{Span\{e_{n}; \, n\geq 0\}}^{\norm{\cdot}_{\alpha,m,R}})^{\perp_{\scal{ \cdot , \cdot}_{\alpha,m,R}}}$$ and therefore
 $$\overline{Span\{e_{n}; \, n\geq 0\}}^{\norm{\cdot}_{\alpha,m,R}}=\ADm.$$
\end{proof}

In order to prove that $\ADm$ is a reproducing kernel Hilbert space, we need to show the following

\begin{lemma}\label{lem-evaluation}
The point evaluation in $\Dr$ is a bounded operation, namely for any fixed $z\in \Dr$, there exists a constant $C_z$ such that
$$ |f(z)| \leq C_z \, \norm{f}_{\alpha,m,R}  $$
for every $f\in \ADm$. Moreover, the mapping $z \mapsto C_z$ is continuous.
\end{lemma}

\begin{proof} For every $ f(z)= \sum\limits_{n=0}^{+\infty}a_{n}z^n \in \ADm$,  we have
\begin{eqnarray*}
|f(z)| \leq  \sum\limits_{n=0}^{+\infty} \left | a_{n}e_{n}(z)\right|
 = \sum\limits_{n=0}^{+\infty} \left(\frac{|e_{n}(z)|}{\norm{e_{n}}_{\alpha,m,R}} \right)  \left(|a_{n}| \norm{e_{n}}_{\alpha,m,R}\right).
\end{eqnarray*}
By the Cauchy-Schwartz inequality, we get
\begin{align*}
|f(z)| \leq  \left(\sum\limits_{n=0}^{+\infty} \frac{|e_{n}(z)|^2}{\norm{e_{n}}_{\alpha,m,R}^2} \right)^{\frac{1}{2}}
     \left(\sum\limits_{n=0}^{+\infty} |a_{n}|^2 \norm{e_{n}}_{\alpha,m,R}^{2} \right)^{\frac{1}{2}}.
\end{align*}
Thence, $f$ satisfies the pointwise estimate $ |f(z)| \leq  C_z \, \norm{f}_{\alpha,m,R},$ where $C_z $ stands for
$$ C_z  := \left(\sum\limits_{n=0}^{+\infty} \frac{|e_{n}(z)|^2}{\norm{e_{n}}_{\alpha,m,R}^2} \right)^{\frac{1}{2}} .$$
Thus,  the evaluation mapping $\delta_{z}: f\mapsto f(z)$ is a continuous linear form on $\ADm$.
\end{proof}

Therefore $\ADm$ is a reproducing kernel Hilbert space by Riesz representation theorem, whose the reproducing kernel function
 is given explicitly in terms of the ${_3F_2}$-hypergeometric function \cite[Chapter 5]{Rainville71},
 $$ {_3F_2}\left( \begin{array}{c} a , b , c \\ a' , b' \end{array}\bigg | x \right)
  = \sum\limits_{k=0}^{+\infty} \frac{(a)_k(b)_k(c)_k}{(a')_k(b')_k} \frac{x^k}{k!}; \quad |x| < 1,$$
  where $(a)_k=a(a+1) \cdots (a+k-1)$ is the Pochhammer symbol.
  Namely, wa have the following.

\begin{proposition}\label{RK-mtext}
 The reproducing kernel of $\ADm$ is given by
\begin{align}\label{expRK-ADm}
K^{\alpha}_{R,m}(z,w) =  \frac{(\alpha+1)}{\pi R^{2}}
\left\{ \sum\limits_{n=0}^{m-1} (\alpha+2)_{n}\frac{(z\overline{w})^{n}}{n!R^{2n}}
+ \frac{(z \overline{w})^{m}}{(m!)^{2}}
{_3F_2}\left( \begin{array}{c} 1 , 1 , \alpha+2 \\ m+1 , m+1 \end{array}\bigg | \frac{z \overline{w}}{R^{2}} \right) \right\}.
\end{align}
\end{proposition}

\begin{proof}
Recall from above that $(z^n)_{n\geq 0}$ is an orthogonal basis of the reproducing kernel Hilbert space $\ADm$. 
Therefore, the reproducing kernel function $K^{\alpha}_{R,m}(z,w)$; $z,w \in \Dr$, of $\ADm$ can be computed by evaluating the sum
$$ K^{\alpha}_{R,m} (z,w) = \sum\limits_{n=0}^{+\infty} \frac{z^n \overline{w}^n } {\norm{e_{n}}_{\alpha,m,R}^{2} } .$$
More explicitly, we have
\begin{align*}
K^{\alpha}_{R,m}(z,w) =  \frac{1}{\pi R^{2}}
\left\{\sum\limits_{n=0}^{m-1}\frac{\Gamma(\alpha+2+n)}{\Gamma(n+1)\Gamma(\alpha+1)}\frac{(z\overline{w})^{n}}{R^{2n}}
 +
 \sum\limits_{n=m}^{+\infty} \frac{\Gamma(n-m+1)\Gamma(\alpha+2+n-m)}{\Gamma^{2}(n+1)\Gamma(\alpha+1)} \frac{(z\overline{w})^{n} }{R^{2(n-m)}}\right\} .
 \end{align*}
By means of $(\alpha+1)\Gamma(\alpha+1)=\Gamma(\alpha+2)$ and the change of index $n-m=p$, we get
\begin{align*}
K^{\alpha}_{R,m}(z,w) =  \frac{(\alpha+1)}{\pi R^{2}} \left\{\sum\limits_{n=0}^{m-1}\frac{\Gamma(\alpha+2+n)}{\Gamma(n+1)\Gamma(\alpha+2)}\left(\frac{z\overline{w}}{R^{2}}\right)^{n}
 +
(z\overline{w})^{m}  \sum\limits_{p=0}^{+\infty}  \frac{\Gamma^2(p+1)\Gamma(\alpha+2+p)}{\Gamma^{2}(p+m+1)\Gamma(\alpha+2)}  \frac{\left(\frac{z\overline{w}}{R^{2}}\right)^{p}}{p!} \right\}.
 \end{align*}
 Finally, since $ \Gamma(p+1) = (1)_p$, $\Gamma(p+m+1)= m! (m+1)_p$ and $\frac{\Gamma(\alpha+2+p)}{\Gamma(\alpha+2)}=(\alpha+2)_p$, it follows
 \begin{align*}
K^{\alpha}_{R,m}(z,w) &=   \frac{(\alpha+1)}{\pi R^{2}}
\left\{\sum\limits_{n=0}^{m-1}\frac{(\alpha+2)_n}{n!}\left(\frac{z\overline{w}}{R^{2}}\right)^{n}
+
\frac{(z\overline{w})^{m}}{(m!)^2}  \sum\limits_{p=0}^{+\infty}  \dfrac{(1)_p(1)_p (\alpha+2)_p}{(m+1)_p(m+1)_p}   \frac{\left(\frac{z\overline{w}}{R^{2}}\right)^{p}}{p!} \right\} \\
&= \frac{(\alpha+1)}{\pi R^{2}} \left\{ \sum\limits_{n=0}^{m-1} (\alpha+2)_{n}\frac{(z\overline{w})^{n}}{n! R^{2n}}
+ \frac{(z \overline{w})^{m}}{(m!)^{2}}
{_3F_2}\left( \begin{array}{c} 1 , 1 , \alpha+2 \\ m+1 , m+1 \end{array}\bigg | \frac{z \overline{w}}{R^{2}} \right) \right\}.
\end{align*}
This completes the proof.
\end{proof}

\begin{remark}\label{Rem-11}
Making use of the fact ${_3F_2}\left( \begin{array}{c} a , a , c \\ a , a \end{array}\bigg | x \right) = (1-x)^{-c}$, we see
   that for the special case of $m=0$, the formula \eqref{expRK-ADm} reads simply
     \begin{align*} 
  K^{\alpha}_{R,0}(z,w)=\frac{(\alpha+1)}{\pi R^{2}} \left( \frac{R^{2}}{R^{2} -  z \overline{w}} \right)^{\alpha+2},
     \end{align*}
  which corresponds to the reproducing kernel of the weighted Bergman space $\ADmo$ on $\Dr$ ((\cite{GhIn2005JMP})). 
\end{remark}

\begin{remark}\label{Rem-12}For $R=1$, $\alpha=0$ and $m=1$, the corresponding reproducing kernel reduces further to be the reproducing kernel the classical Dirichlet space,
\begin{align*} 
K^{\alpha}_{R,1}(z,w) = \frac{1}{\pi }\left( 1 +  z \overline{w} \, {_2F_1}\left( \begin{array}{c} 1 , 1  \\  2 \end{array}\bigg | z \overline{w} \right)\right)
=   \frac{1}{\pi }\left( 1 + \ln\left( \frac 1{1-z}\right)\right)
\end{align*}
thanks to  the transformation \cite[p. 109]{Temme1996} 
\[ x {_2F_1}\left( \begin{array}{c} 1 , 1  \\ 2 \end{array}\bigg | x \right) =
 \mathop{\ln\/}\nolimits\! \left(\frac{1}{1-x}\right).\]
\end{remark}

We conclude this section by noting that the proof of Theorem \ref{Mthm1}, stated in the introductory section,
 is contained in the previous established results (essentially Lemmas \ref{lem-evaluation} and \ref{lem-growth},
 and Propositions \ref{Hilbert} and \ref{RK-mtext}).

\section{Weighted Bargmann-Dirichlet spaces of order $m$ on the complex plane}

Fix a real number $\nu > 0$ and denote by $\LC$ the usual Hilbert space of all square-integrable functions on $\C$ with respect to the Gaussian mesure $\nuar(z)= e^{-\nu|z|^{2}}\mes(z)$.  The hermitian inner product is defined by
\begin{align}
\scal{ f , g}_{\nu}:= \int_{\C} f(z) \overline{g(z)}  \nuarexp(z),
\end{align}
and the associated norm by
\begin{align}
\norm{f}^{2}_{\nu}:= \int_{\C}|f(z)|^{2} \nuar(z).
\end{align}
For fixed nonnegative integer $ m=0,1,2, \cdots$, any $f(z)=\sum\limits_{n=0}^{+\infty}a_{n}z^{n}$ in the vector space $\HolC$
 of all convergent entire series on $\C$, can be written as $f(z)=f_{1,m}(z)+f_{2,m}(z) $, where
$$f_{1,m}(z):=\sum\limits_{n=0}^{m-1}a_{n}z^{n}  \quad\quad \mbox{and} \quad   f_{2,m}(z):=\sum\limits_{n=m}^{+\infty}a_{n}z^{n}. $$
Then, one can perform the functional space $\BDm$ of all entire functions $f\in \HolC$ such that
$$ \norm{f}_{\nu,m}^{2}:=\norm{f_{1,m}}_{\nu}^{2}+\norm{f^{(m)}_{2,m}}_{\nu}^{2}<+\infty.$$
Notice that the hermitian inner product $\scal{ \cdot , \cdot}_{\nu,m}$ associated to the norm $\norm{\cdot}_{\nu,m}$ is given through
\begin{align} \label{relation-mo}
\scal{ f , g}_{\nu,m} : = \scal{ f_{1,m} , g_{1,m}}_{\nu}+\scal{ f_{2,m}^{(m)} , g_{2,m}^{(m)}}_{\nu} .
\end{align}

\begin{lemma} The monomials $e_{n}(z)=z^{n}$; $n\geq 0$, belong to $\BDm$ and are pairwise orthogonal with respect to the hermitian scalar product
$\norm{\cdot}_{\nu,m}$, with
   \begin{align} \label{nk-spc}
   \norm{e_{n}}_{\nu,m}^{2} = \pi \left\{\begin{array}{ll}
                                    \frac{n!}{\nu^{n+1}}                    &\quad \mbox{for } \,   n<m \\
                                    \frac{(n!)^2}{\nu^{n-m+1}\Gamma(n-m+1)} &\quad \mbox{for } \, n\geq m
                                  \end{array} \right. .
   \end{align}
   Moreover, a function $f(z)=\sum\limits_{n=0}^{+\infty}a_{n}z^{n} $ belongs to $\BDm$ if and only if
   \begin{align} \label{normefc}
     \norm{f}_{\nu,m}^{2}=\pi \left\{ \sum\limits_{n=0}^{m-1}\left(\frac{n!}{\nu^{n+1}}\right)\mid a_{n}\mid ^{2} +\sum\limits_{n=m}^{+\infty}\left( \frac{(n!)^{2}}{\nu^{n-m+1}(n-m)!}\right)\mid a_{n}\mid ^{2} \right\}<+\infty .
   \end{align}
\end{lemma}

\begin{proof}
By definition \eqref{relation-mo} of $\norm{\cdot}_{\nu,m}$, it is not difficult to see that $\displaystyle \scal{ e_{n} , e_{k}}_{\nu,m} : = \scal{ e_{n} , e_{k}}_{\nu}$ when $n,k<m $, $\scal{ e_{n} , e_{k}}_{\nu,m} : = \frac{n!k!}{(n-m)!(k-m)!}\scal{e_{n-m} , e_{k-m}}_{\nu}$ when $m\leq n,k $ and
$\scal{ e_{n} , e_{k}}_{\nu,m} =0$ otherwise. Thus, the first assertion follows making use of the well established formula 
$$\int_{\C} z^n \overline{z}^k \nuarexp
= \frac{\pi n!}{\nu^{n+1}}  \delta_{n,k},$$
which can be handled using polar coordinates $z=re^{i\theta}$ and the change $t=\nu r^2$, combined with the known facts
$$\int_0^{2\pi} e^{i(n-k)\theta} d\theta = 2\pi \delta_{n,k} \quad \mbox{and} \quad  \int_0^{+\infty} t^{n} e^{-t} dt =n!.$$
Finally, \eqref{normefc} follows by orthogonality of the monomials in $\BDm$ keeping in mind the explicit expression of $\norm{e_{n}}_{\nu,m}^{2}$
given through \eqref{nk-spc} and the fact that the series $f$ belongs to  $\BDm$ if and only if  $\norm{f}_{\nu,m}^{2}$ is finite.
\end{proof}

\begin{remark}
For $m=0$, the considered space is to the classical Bargmann-Fock Hilbert space
$\BDmo:=L^{2}\left(\C;\nuarexp\right)\cap \HolC.$ While for $m=1$, it reads simply
$$\BDmmo:=\bigg\{f\in \HolC, \quad \quad\dfrac{\pi}{\nu}|f(0)|^{2}+\int_{\C}|f'(z)|^{2} \nuarexp(z) <+\infty \bigg\}.$$
\end{remark}

\begin{definition}
We call $\BDm$ weighted Bargmann-Dirichlet spaces of order $m$.
\end{definition}

\begin{lemma}
   For every fixed nonnegative integer $m$, the space $\BDmm$ is continuously embedded in $\BDm$. More precisely, there exists a constant
   $C_{\nu,m}$ depending only in $\nu$ and $m$ such that
  $$\norm{f}^{2}_{\nu,m}\leq C_{\nu,m} \norm{f}^{2}_{\nu,m+1}.$$
   In particular, the weighted Bargmann-Dirichlet space $\BDm$ is continuously embedded in the classical Bargmann-Fock space.
 \end{lemma}

 \begin{proof}
 It is similar to the one given for Lemma \ref{lem-embedding}.
 \end{proof}

Thanks to the previous obtained results, one can proceed exactly as in the proof of Proposition \ref{Hilbert} and Lemma \ref{lem-evaluation} to show the  following

\begin{proposition} $\BDm $ is a Hilbert space and the monomials $e_{n}(z)=z^{n}; n\geq 0$ constitute an orthogonal basis of it.
Moreover, the evaluation map $\delta_{z}: f\mapsto f(z)$, for fixed $z\in \C $, is a continuous linear form  on $\BDm$ and satisfies
$$ |f(z)| \leq C_{z} \, \norm{f}_{\nu,m}$$
for every $f\in \BDm$, where $$C_{z}=\left(  \sum\limits_{n=0}^{+\infty}  \frac{\mid z^n \mid ^2}{\norm{e_{n}}_{\nu,m}^2}\right)^{\frac{1}{2}}   .$$
\end{proposition}

This shows that $\BDm$ is a reproducing kernel Hilbert space. Its reproducing kernel function
 is given explicitly in terms of the ${_2F_2}$-hypergeometric function (\cite[Chapter 5]{Rainville71}),
 $$ {_2F_2}\left( \begin{array}{c} a , b  \\ a' , b' \end{array}\bigg | x \right)
  = \sum\limits_{k=0}^{+\infty} \frac{(a)_k(b)_k}{(a')_k(b')_k} \frac{x^k}{k!}. $$
 More precisely, we assert

\begin{proposition}
The reproducing kernel of space $\BDm$ is given by
\begin{align}\label{RK-BDC} K_{m}^{\nu}(z,w)=\frac{\nu}{\pi}\left(
  \sum\limits_{n=0}^{m-1}\frac{\nu^n(z \overline{w})^{n}}{n!}+\frac{(z \overline{w})^{m}}{(m!)^2} \,
   {_2F_2}\left(\begin{array}{c} 1, 1\\ m+1,m+1 \end{array}\bigg |  \nu z \overline{w}\right)\right) .
   \end{align}
\end{proposition}

\begin{proof}
We have
\begin{align*}
K_{m}^{\nu}(z,w) &=  \sum\limits_{n=0}^{+\infty} \frac{z^n \overline{w}^n } {\norm{e_{n}}_{\nu,m}^{2} } \\
& =\frac{\nu}{\pi }\left( \sum\limits_{n=0}^{m-1} \frac{\nu^n(z\overline{w})^{n} }{n!}+
 \sum\limits_{n=m}^{+\infty} \frac{(n-m)!\nu^{n-m}(z\overline{w})^{n}}{(n!)^2}\right)\\
&= \frac{\nu}{\pi }\left( \sum\limits_{n=0}^{m-1} \frac{\nu^n(z\overline{w})^{n} }{n!}+
\frac{(z\overline{w})^{m}}{\Gamma^{2}(m+1)} \sum\limits_{n=m}^{+\infty} \frac{\Gamma^{2}(m+1)\Gamma^{2}(n-m+1)}{\Gamma^{2}(n+1)}\frac{(\nu z\overline{w})^{n-m}}{(n-m)!}  \right)\\
&= \frac{\nu}{\pi}\left(  \sum\limits_{n=0}^{m-1}\frac{\nu^n (z \overline{w})^{n}}{n!}+\frac{(z \overline{w})^{m}}{(m!)^2} \,
   {_2F_2}\left(\begin{array}{c} 1, 1\\ m+1,m+1 \end{array}\bigg |  \nu z \overline{w}\right)\right).
\end{align*}
This completes the proof
\end{proof}

\begin{remark}
For $m=0$, we recover the reproducing kernel function of the Bargmann-Fock space which is known to be given by
  $$ K^{\nu}(z,w)=\left(\frac{\nu}{\pi}\right) e^{\nu z \overline{w}}.$$
\end{remark}

\section{Weighted Bargmann-Dirichlet spaces as limit of weighted Bergman-Dirichlet spaces}

The complex space $\C$ endowed with the flat  metric $ ds^2_\infty =
dz\otimes d{\bar z}$ can be seen as a K\"ahlerian manifold.
It is shown in \cite{GhIn2005JMP} that the flat hermitain geometry on $\C$
can be approximated by the complex hyperbolic geometry of the disks $\Dr$ of radius $R>0$
associated to the scaled Bergman K\"ahler metric
$$ ds^2_R = \frac{R^4 }{(R ^2-|z|^2)^2} dz \otimes d{\bar z}.$$
In fact, the holomorphic sectional curvature $\kappa_R$ of the complete K\"ahlerian
manifold $(\Dr,ds^2_R)$, which is known to be a negative constant, $\kappa_R=- 4/R^2$, tends to $0$, the flat curvature of
$(\C,ds^2_\infty)$. Moreover,  the measure $\muar$ on $\Dr$ is the one associated to the metric $ds^2_R$. It converges to
the volume measure associated to $ds^2_\infty$, when $R$ goes to $+\infty$, being indeed
$$  \lim_{R\rightarrow +\infty} \left( 1- \left|\frac{z}R\right|^2\right) ^{\nu R^2}  \mes = e^{-\nu \vert z\vert^{2}}  \mes.$$
Thus we have instead of general $\alpha >-1$, we consider the particular case of $\alpha =\nu R^2$ with $\nu>0$, so that
$$\displaystyle \lim\limits_{R\rightarrow +\infty} \muar = \nuar.$$

The main result of this section concerns the pointwise convergence of the reproducing kernel functions.

\begin{theorem} \label{thm-lim}
Let $K^{\alpha}_{R,m}$ with $\alpha = \nu R^{2}$ (resp. $K^{\nu}_{m}$) be the the reproducing kernel of the weighted Bergman-Dirichlet (resp. Bargmann-Dirichlet) space of order $m$. Then, we have
$$  \lim_{R\rightarrow +\infty} K^{\nu R^2}_{R,m}(z,w) = K^{\nu}_{m}(z,w)$$
for every fixed $(z,w) \in \C\times \C.$
\end{theorem}

The proof of this theorem, follows by making use of the explicit expression of the reproducing kernels
$K^{\nu R^2}_{R,m}$ and  $K^{\nu}_{m}$ as given by \eqref{RK-mtext} and \eqref{RK-BDC}, respectively, combined with the following lemma

\begin{lemma}\label{LemLimitHyp}
 For every fixed $\xi\in\C$, we have
\begin{eqnarray}
\lim\limits_{\rho \longrightarrow +\infty} {_{p+1}F_p}\left(\begin{array}{c} a_1, \cdots, a_p, c+\rho \\ a_1', \cdots, a_p' \end{array}\bigg |  \frac{\xi}{\rho}\right)=
{_pF_p}\left(\begin{array}{c} a_1, \cdots, a_p \\ a_1', \cdots, a_p' \end{array}\bigg |  \xi \right),
 \label{limhyp}
\end{eqnarray}
where $a_j, a_j'$; $j=1, \cdots, p$, are complex numbers.
Moreover, the convergence is uniform on compact sets of $\C^n$
\end{lemma}

\begin{proof}
It can be checked easily in a formal way. For a rigorous proof, one can proceed exactly as in \cite{GhIn2005JMP} for $p=1$.
\end{proof}

 \begin{remark}  According to Lemma \ref{LemLimitHyp}, the convergence in Theorem \ref{thm-lim} of the reproducing kernel functions is
uniform in $z,w$ for $z,w$ in any compact set of $\C\times \C$.
\end{remark}

\noindent{\bf Acknowledegement:}
The assistance of the members of the seminars "Partial differential equations and spectral geometry" is gratefully acknowledged.
The second and the third authors are partially supported by the Hassan II Academy of Sciences and Technology.

\end{document}